\documentclass[3p,times]{amsart}

\usepackage[left=2cm, right=2cm, top=2cm]{geometry}
\usepackage[foot]{amsaddr}
\usepackage{amsmath,amssymb,amsthm,mathtools,mathdots,nicefrac,xcolor,stmaryrd}

\theoremstyle{definition} 
\newtheorem{theorem}{Theorem}[section]

\newtheorem{lemma}[theorem]{Lemma}
\newtheorem{proposition}[theorem]{Proposition}

\newtheorem{example}[theorem]{Example}

\newcommand{\B}{\mathcal{B}}
\newcommand{\V}{\mathcal{V}}

\newcommand{\cart}{\boxempty}

\DeclareMathOperator{\val}{val}
\DeclareMathOperator{\tw}{tw}
\DeclareMathOperator{\bn}{bn}
\DeclareMathOperator{\sbn}{sbn}
\DeclareMathOperator{\gon}{gon}

\DeclarePairedDelimiter\abs{\lvert}{\rvert}
\DeclarePairedDelimiter\ceil{\lceil}{\rceil}

\DeclarePairedDelimiter\norm{\lVert}{\rVert}
\DeclarePairedDelimiter\set{\{}{\}}
\DeclarePairedDelimiter\paren{(}{)}

\makeatletter
\let\oldabs\abs
\def\abs{\@ifstar{\oldabs}{\oldabs*}}
\let\oldnorm\norm
\def\norm{\@ifstar{\oldnorm}{\oldnorm*}}
\let\oldparen\paren
\def\paren{\@ifstar{\oldparen}{\oldparen*}}
\makeatother

\author{Ivan Aidun, Frances Dean, Ralph Morrison, Teresa Yu, and Julie Yuan}

\begin{document}


\title{Treewidth and gonality of glued grid graphs}

\maketitle

\begin{abstract}
We compute the treewidth of a family of graphs we refer to as the glued grids, consisting of the stacked prism graphs and the toroidal grids. Our main technique is constructing strict brambles of large orders. We discuss connections to divisorial graph theory coming from tropical geometry, and use our results to compute the divisorial gonality of these graphs.\end{abstract}

\section{Introduction}

The \emph{treewidth} of a graph is a measure of the graph's similarity to a tree, first introduced by Robertson and Seymour in \cite{rs}.  Treewidth is a natural and powerful measure of a graph's complexity: there are polynomial time algorithms for many difficult problems on graphs of bounded treewidth \cite{bod}.  A well-known result is that an $n \times n$ grid graph has treewidth $n$ \cite{bod2}, implying that the treewidth of planar graphs is unbounded; this is in contrast to similar graph parameters such as the Hadwiger number, which is bounded for planar graphs \cite{wag}.  Since treewidth is minor-monotone \cite{hal}, the treewidth of a graph is at least the size of its largest grid minor.  For these reasons, studying grid graphs is important in the study of treewidth.

\begin{figure}[hbt]
    \centering
    \includegraphics{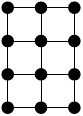}\quad\quad \includegraphics{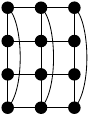}\quad\quad \includegraphics{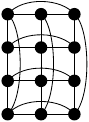}
    \caption{The $4\times 3$ grid, stacked prism, and toroidal grid, respectively}
    \label{grid_examples}
\end{figure}

In this paper, we are interested in certain natural generalizations of the grid graphs, which we collectively refer to as the glued grids:  the \emph{stacked prism graphs} $Y_{m,n}=C_m \cart P_n$ and the \emph{toroidal grid graphs} $T_{m,n}=C_m \cart C_n$.   Here $P_n$ is the path on $n$ vertices, $C_n$ is the cycle on $n$ vertices, and $G\cart H$ denotes the Cartesian product of $G$ and $H$. As illustrated in Figure \ref{grid_examples}, the glued grids resemble grids with additional edges: both the stacked prisms and the toroidal grids have edges wrapping from the top row to the bottom row, and  the toroidal grids also have edges wrapping from the leftmost column to the rightmost.  In this sense, we can see them as grids glued along their boundaries, analogous to topological quotients on the boundary of the unit square.

It was shown in \cite{koo} that the treewidth of the $n \times n$ toroidal grid is either $2n - 2$ or $2n - 1$, and that both of these values are achieved for certain $n$.  They also give the following general bounds: \[\min\set{m,n} \leq \tw(Y_{m,n}) \leq \min\set{m,2n}\]
and
\[2\min\set{m,n}-2 \leq \tw(T_{m,n}) \leq 2 \cdot \min\set{m,n}.\]  We use strict brambles to explicitly compute the treewidth of glued grids in all but three exceptional cases, namely $Y_{2n,n}$, $T_{n+1,n}$, and $T_{n,n}$, the third of which was already studied in \cite{koo}.  Our two main treewidth results are the following, with the main result of \cite{koo} included as the third case of Theorem \ref{theorem:toroidal}.

\begin{theorem}  The treewidth of the stacked prism $Y_{m,n}$ is
\[\tw(Y_{m,n})=\begin{cases}\min\{m,2n\} &\textrm{if $m\neq 2n$}
\\ \min\{m,2n\}-1\textrm{ or }\min\{m,2n\} &\textrm{if $m= 2n$}. \end{cases}\]
If $m=2n$, both possible treewidths occur for certain values of $n$.
\label{theorem:stackedprism}
\end{theorem}

\begin{theorem}  The treewidth of the toroidal grid $T_{m,n}$ is
\[\tw(T_{m,n})=\begin{cases}2\min\{m,n\} &\textrm{if $|m-n|\geq2$}
\\ 2\min\{m,n\}-1\textrm{ or }2\min\{m,n\} &\textrm{if $|m-n|=1$}
\\ 2n-2\textrm{ or }2n-1 &\textrm{if $m=n$}.\end{cases}\]
If $|m-n|=1$ or $m=n$, both possible treewidths occur for certain values of $m$ and $n$.
\label{theorem:toroidal}
\end{theorem}

Our original motivation for this project was to understand \emph{divisorial gonality}, another invariant of these glued graphs.  In Section \ref{section:gonality}, we discuss connections between gonality and treewidth, and we compute the gonality of all glued grids except for $Y_{2n,n}$, $T_{n+1,n}$, and $T_{n,n}$ in Theorems \ref{theorem:stackedprismgon} and \ref{theorem:toroidalgon}.

Our paper is organized as follows. In Section \ref{section:definitions}, we review background and definitions and establish an explicit link between treewidth and strict brambles.  In Sections \ref{section:cylinder} and \ref{section:toroidal}, we prove Theorems \ref{theorem:stackedprism} and \ref{theorem:toroidal}, respectively.  Section \ref{section:gonality} consists of our results on divisorial gonality.

\section{Treewidth, brambles, and strict brambles}
\label{section:definitions}

Throughout this paper we let $G=(V,E)$ be a connected, simple graph with vertex set $V$ and edge set $E$. When multiple graphs are under discussion, we use $V(G)$ and $E(G)$ to denote the vertices and edges of $G$, respectively.  A subset $A \subset V(G)$ is said to be \emph{connected} if the subgraph it induces is a connected graph.  Given two graphs $G_1 = (V_1,E_1)$ and  $G_2 = (V_2, E_2)$, the \emph{Cartesian product} $G_1 \cart G_2$ is the graph with vertex set $V(G_1\times G_2)=V_1 \times V_2$ and an edge between $(v_1,v_2)$ and $(u_1,u_2)$ if
\begin{itemize}
    \item $v_2 = u_2$ and $v_1u_1 \in E_1$, or
    \item $v_1 = u_1$ and $v_2u_2 \in E_2$.
\end{itemize}
With this definition, we see that the $m \times n$ grid graph $G_{m,n}$ is the Cartesian product $P_m \cart P_n$ of two path graphs; the stacked prism graph $Y_{m,n}$ is the Cartesian product $C_m \cart P_n$ of a cycle graph and a path; and the toroidal grid graph $T_{m,n}$ is the Cartesian product $C_m \cart C_n$ of two cycles.  To ensure these graphs are simple, we assume $m\geq 3$ for $Y_{m,n}$ and that $m,n\geq 3$ for $T_{m,n}$.

Given a graph $G$, let $T$ be a tree, and let $\V = \set{V_t}$ be a family of subsets $V_t \subset V(G)$ indexed by the nodes $t \in T$.  The pair $(T,\V)$ constitutes a \emph{tree decomposition} of $G$ if it satisfies the following:
\begin{itemize}
    \item[(1)] $\displaystyle{\bigcup_{t\in T} V_t = V(G)}$;
    \item[(2)] if $uv \in E(G)$ then there is a node $t \in T$ such that $u,v \in V_t$; and
    \item[(3)] if $v \in V_{t_1},V_{t_2}$, then for every $t'$ in the (unique) path from $t_1$ to $t_2$ in $T$, we have $v \in V_{t'}$.
\end{itemize}

The sets $V_t$ are commonly referred to as \emph{bags}.  The \emph{width} of a tree decomposition is one less than the size of the largest bag: $w(T,\V) = \max_{t\in T}(\abs{V_t}-1)$.  The \emph{treewidth} of $G$, denoted $\text{tw}(G)$, is the minimum width of a tree decomposition of $G$.

One useful {property} of a tree decomposition is its connection to separating sets in the starting graph.  Given subsets $S$, $U_1,$ and $U_2$ of $V(G)$, we say $S$ \emph{separates} $U_1$ from $U_2$ if every path from an element of $U_1$ to an element of $U_2$ includes an element of $S$.  Note that if $S$ separates $U_1$ and $U_2$, then $U_1\cap U_2\subseteq S$.

\begin{lemma}[Lemma 12.3.1 in \cite{die}]\label{lemma:separating}   Let $(T,\V)$ be a tree decomposition of a graph $G$.  Let $t_1t_2$ be any edge of $T$, and let $T_1$ and $T_2$ be the two connected components of $T-t_1t_2$, where $t_i\in T_i$.  If $U_i=\bigcup_{t\in T_i}V_t$, then $V_{t_1}\cap V_{t_2}$ separates $U_1$ from $U_2$ in $G$.
\end{lemma}

Our main tools for computing treewidth are structures called brambles.  We say two subsets $A,B \subset V(G)$ \emph{touch} if either they share a vertex or there is an edge between them.  A \emph{bramble} is a family of connected, mutually touching vertex sets.  If for all $B_1,B_2 \in \B$ we have $B_1 \cap B_2 \neq \emptyset$, then the bramble $\B$ is called \emph{strict}.  A set $S \subset V(G)$ is said to \emph{cover} the bramble $\B$ if $S \cap B \neq \emptyset$ for all $B \in \B$.  A set which covers $\B$ is called a \emph{hitting set} for $\B$.  The \emph{order of a bramble $\B$}, written $\norm{\B}$, is the minimum size of a hitting set for $\B$. 

\begin{example}[From Section 4.5 in \cite{db}]
\label{example:bramble}
Consider the $m\times n$ grid $G_{m,n}$.  Let $\mathcal{B}=\{B_{ij}\}$, where $B_{ij}$ is the union of the $i^{th}$ column with the $j^{th}$ row.  This is a strict bramble:  $B_{ij}$ and $B_{k\ell}$ intersect at the intersection of the $i^{th}$ column with the $\ell^{th}$ row.  It turns out that $\norm{\B}=\min\{m,n\}$.  To see this, assume without loss of generality that $m\leq n$.   If $S$ consists of the $m$ vertices in a column, then $S$ is a hitting set, as  each $B_{ij}$ intersects every column in at least one point.  So, $\norm{B}\leq m$.  However, any set $T$ with $|T|\leq m-1$ is not a hitting set, as $T$ must then miss some row and some column, and so there exists some $B_{ij}$ such that $B_{ij}\cap T=\emptyset$.  It follows that $\norm{\B}\geq m$, and we conclude that $\norm{\B}= m=\min\{m,n\}$.

\end{example}

The \emph{bramble number (resp. strict bramble number)} of a graph is the maximum order of a  bramble (resp. strict bramble) on $G$.  We denote the bramble number $\bn(G)$ and the strict bramble number $\sbn(G)$.  We can immediately see that $\sbn(G)\leq \bn(G)$, since every strict bramble is a bramble.

\begin{example}  For an $m\times n$ grid graph, we have $\bn(G_{m,n})=\tw(G_{m,n})+1=\min\{m,n\}+1$ by \cite{bod2} and Proposition \ref{bramble bound} below. On the other hand, $\sbn(G_{m,n})\geq\min\{m,n\}$ by Example \ref{example:bramble}.  Since $\sbn(G_{m,n})\leq \bn(G_{m,n})$, the bramble number and the strict bramble number differ by at most $1$.  (It will follow from Lemma \ref{lemma:mathoverflow} that they cannot be equal, so they differ by exactly $1$.)
\end{example}

It was shown in  \cite[\S 2.2]{koy} that $\bn(G)\leq 2\sbn(G)$, so the bramble number and the strict bramble number never differ by more than a factor of $2$. This bound is sharp, due to the following example.

\begin{example} For a complete graph on $n$ vertices we have $\bn(K_n)=n$, since the family of all one-vertex subsets of $V(G)$ is a bramble of order $n$, and it has the largest possible order of any bramble on any graph with $n$ vertices.  However, $\sbn(K_n)= \ceil{n/2}$.  The fact that $\sbn(K_n)\geq \ceil{n/2}$ follows from the fact that  $n=\bn(G)\leq 2\sbn(G)$.  To see that $\sbn(K_n)\leq \ceil{n/2}$,  suppose for the sake of contradiction that $\mathcal{B}$ is a strict bramble on $K_n$ with $\norm{\B}>\ceil{n/2}$.  Then for any choice of $\ceil{n/2}$ points, there exists an element of $\mathcal{B}$ consisting of vertices not including those points.  Labelling the vertices of $K_n$ as $v_1,\ldots,v_n$, this means that there is an element of $\B$ not hit by $\{v_1,\ldots,v_{\lceil{n/2}\rceil}\}$, and also an element of $\B$ not hit by $\{v_{\lceil n/2\rceil+1},\ldots,v_{n}\}$.  However, two such elements of $\mathcal{B}$ cannot intersect, contradicting $\mathcal{B}$ being a strict bramble.  We conclude that $\sbn(K_n)=\ceil{n/2}$.
\end{example}

 We remark that \cite{koy} refers to the strict bramble number as the \emph{pairwise intersecting number}.  They use this number to provide a lower bound on the treewidth of product graphs.  Unfortunately, their lower bounds are  $\tw(Y_{m,n}) \geq 3$ and $\tw(T_{m,n}) \geq 5$, which is the desired lower bound only for $Y_{3,2}$.

The utility of brambles in studying treewidth comes from a theorem of Seymour and Thomas stating that $\tw(G) < k$ if and only if $G$ does not admit a bramble of order greater than $k$ \cite{st}.  In other words, we have $\tw(G) = \bn(G) - 1$.  While the reverse direction of their theorem is quite involved, the forward direction is straightforward.  We reproduce a proof here for the reader's convenience, based on the proof presented in \cite[Theorem 12.4.3]{die}.

\begin{proposition}\label{bramble bound}
Let $G$ be a graph.  If $\tw(G) < k$, then $G$ does not admit a bramble of order greater than $k$.
\end{proposition}
\begin{proof}
Let $\B$ be a bramble and let $(T,\V)$ be a tree decomposition of $G$.  We will show that one of the bags $V_t$ covers $\B$.  For every edge $t_1 t_2 \in E(T)$, if $X \coloneqq V_{t_1} \cap V_{t_2}$ covers $\B$, then we are done.  Otherwise, $T - t_1 t_2$ separates the vertices of $G$ into two sets, which we label $U_1$ and $U_2$ as in Lemma \ref{lemma:separating}.  Now, every $B \in \B$ not hit by $X$ must fall into either $U_1\setminus X$ or $U_2\setminus X$, and in fact they must all fall into the same set since these two sets do not touch.  If they fall into $U_1$, then we orient the edge $t_1 t_2$ toward $U_1$, and similarly if they fall into $U_2$.

We orient all of the edges of $T$ in the same manner, and let $t$ be a node of $T$ which is the end of a maximal directed path.  Then $V_t$ covers $\B$.
\end{proof}

Thus, we know that for any bramble $\B$, $\tw(G) \geq \norm{\B} - 1$, so we can lower bound treewidth by constructing a bramble of large order. Indeed, this is the main technique in \cite{koo}.  The next claim shows that we can omit the $-1$ when the bramble we construct is strict.  We thank Jan Kyncl on MathOverflow for communicating the following proof to us.

\begin{lemma}
\label{lemma:mathoverflow}
For any graph $G$, we have $\tw(G) \geq \sbn(G)$.
\end{lemma}
\begin{proof}
Let $\B$ be a strict bramble and $(T,\V)$ be a tree decomposition of a graph $G$. The proof of Proposition \ref{bramble bound} shows that at least one of the following must be true:
\begin{itemize}
    \item[(i)] There is an edge $t_1 t_2$ of $T$ such that $V_{t_1} \cap V_{t_2}$ covers $\B$. 
    \item[(ii)] There is a bag $V_t$ which covers $\B$.
\end{itemize}
If (i) holds, then since we may assume the $V_t$ are pairwise distinct, $\norm{\B}$ is at most the width of $(T,\V)$.  Now assume that (ii) holds and (i) does not.  Since $V_t$ may have $\tw(G) + 1$ vertices, we want to show that a proper subset of $V_t$ will also cover $\B$.  If $G$ has $\abs{V_t}$ vertices, then we can omit any one vertex from $V_t$ to still retain a cover of $\B$, or else $\B$ contains a singleton and has order 1. 

Since $G$ is connected, if $G$ has more than $\abs{V_t}$ vertices, then there is an edge $tt'$ in $T$ such that $\abs{V_t \cap V_{t'}} \geq 1$. By Lemma \ref{lemma:separating}, the set $S=V_t\cap V_{t'}$ separates the sets of vertices of $G$ induced by the two components of $T-tt'$; call these $U$ on the $t$ side, and $U'$ on the $t'$ side.  Note that $S=U\cap U'$.  We claim that $V_t \setminus V_{t'}$ covers $\B$.  Suppose not; then there must be a set $B_1 \in \B$ disjoint from $V_t \setminus V_{t'}$.  
Since $V_t$ covers $\B$, we know $B_1$ contains a vertex in $V_t\setminus(V_t\setminus V_{t'})=V_t\cap V_{t'}=S$.  By Lemma \ref{lemma:separating} and the structure of a tree decomposition, any path from a vertex in $S$ to a vertex in $U\setminus S$ must contain a vertex in $V_t\setminus V_{t'}$.  Since $B_1$ has a vertex in $S$ but not in $V_t\setminus V_{t'}$, it follows that $B_1\subseteq (U\setminus S)^C=U'$. However, by the assumption that $S$ does not cover $\B$, there must be a set $B_2 \in \B$ disjoint from $S$. We know that $B_2$ intersects $V_t\subset U$ in a vertex $v$, so if $B_2$ intersects $U'$ in a vertex $w$, then since $B_2$ is connected there exists a $vw$-path from $v\in U$ to $w\in U'$ not passing through $S$, which is impossible by Lemma \ref{lemma:separating}.  Thus we have that no such $w$ exists, so $B_2\subseteq (U')^C$.  It follows that $B_1$ and $B_2$ do not intersect, a contradiction to $\B$ being a strict bramble.
\end{proof}

We note that an immediate corollary of this result is that $\bn(G) > \sbn(G)$.  

\section{Brambles for the stacked prism}
\label{section:cylinder}

In this section, we consider $Y_{m,n}$, the $m\times n$ stacked prism graph with $m$ rows and $n$ columns, glued along the $n$-side.  We first present strict brambles of order $\min\{m,2n\}$ in the cases when $2n\neq m$ to achieve a lower bound on treewidth.  We then argue a lower bound of $2n-1$ on treewidth  when $m=2n$, although this bound is not sharp in all cases.  We  combine these with upper bounds on treewidth to achieve our desired results from Theorem \ref{theorem:stackedprism}.

For the case when $2n<m$, consider the family $\mathcal{B}_1$ consisting of all subgraphs of $Y_{m,n}$ made up of a column with a single vertex deleted, together with two rows (neither intersecting the column in the deleted point).  An element of this family is shown in Figure \ref{gg_bramble_2n<m}.

\begin{figure}[hbt]
    \centering
    \includegraphics{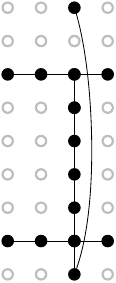}
    \caption{An element of $\mathcal{B}_1$}
    \label{gg_bramble_2n<m}
\end{figure}

\begin{proposition}  If $2n<m$, then the family $\mathcal{B}_1$ is a strict bramble of order $2n$ on $Y_{m,n}$.
\label{prop:m>2n}
\end{proposition}

\begin{proof}  First we show that $\mathcal{B}_1$ is a strict bramble.  Every element of $\B_1$ is connected, by construction.  Any two elements of $\B_1$ intersect in at least one vertex since the column of one is either the same as the column of the other (perhaps with a different point deleted), or will intersect at least one of the two rows of the other.  Thus $\mathcal{B}_1$ is a strict bramble.

We  now show that $\norm{\B_1}=2n$.  Let $S\subset V(Y_{m,n})$ be a subset of size $2n-1$.  Since $2n<m$, there are at least two rows that do not intersect $S$.  Similarly, since there are $n$ columns, at least one column has no more than one element of $S$.  Build an element of $\B_1$ out of these two rows and out of this column, with a point removed at the same location as the element of $S$ (if it exists in that column; otherwise any point not in one of the rows may be removed).  This graph does not intersect $S$, so $S$ cannot be a hitting set of $\B_1$.
This means $\norm{\B_1}\geq 2n$.

To see that  $\norm{\B_1}\leq 2n$, let $S$ be the collection of all vertices in the first two rows. Every column intersects $S$ in two points,  and since each element of $\B$ contains all but one of the points in a column, every element of $\B$ intersects $S$.  We conclude that $\norm{\B}=2n$.
\end{proof}

The second bramble we present is for the case when $m<2n$.  We construct $\mathcal{B}_2=\mathcal{C}\cup\mathcal{D}\cup\mathcal{E}$ as follows.  In all cases, we forbid the deletion of the intersection vertex of a row and a column.

\begin{itemize}
    \item An element of $\mathcal{C}$ is the union of a row and a column.
    \item  An element of $\mathcal{D}$ is the union of one row and two columns, where each of columns has a point removed, each from a different row.
    \item  An element of $\mathcal{E}$ is a union of two rows and two columns, where both the columns have the vertex removed from the same row.
\end{itemize}

See Figure \ref{gg_bramble_2n>m} for illustrations.

\begin{figure}[hbt]
    \centering
    \includegraphics{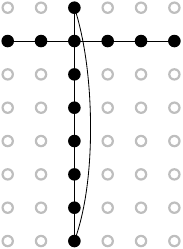}\quad\quad \includegraphics{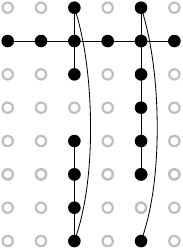}\quad\quad \includegraphics{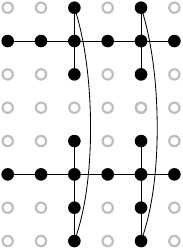}
    \caption{Elements of $\mathcal{C}$, $\mathcal{D}$, and $\mathcal{E}$, respectively}
    \label{gg_bramble_2n>m}
\end{figure}

\begin{proposition}  If $m<2n$, then the family $\mathcal{B}_2$ is a strict bramble of order $m$ on $Y_{m,n}$.
\label{prop:m<2n}
\end{proposition}

\begin{proof} First we show that $\B_2$ is a strict bramble.   Note that each set is connected by construction.  To see that each pair of elements of $\B_2$ intersects, first note that any element of $\mathcal{C}$ contains an entire column, and any element of $\mathcal{B}_2$ has an entire row, so any such pair must intersect.  Similarly, any element of $\mathcal{D}$ intersects any complete row, so it intersects any element of $\mathcal{B}_2$.  Finally, to see that any element of $\mathcal{E}$ intersects any element of $\mathcal{E}$, note that any pair of complete rows intersects any element of $\mathcal{E}$.  Thus, $\mathcal{B}_2$ is a strict bramble.

We now show that $\norm{\B_2}=m$.  Suppose $S\subset V(Y_{m,n})$ with $|S|=m-1$.  We know that at least one row does not intersect $S$; call this the $i^{th}$ row.  Suppose for a moment that some column does not intersect $S$; then this column, together with the $i^{th}$ row, forms an element of $\mathcal{C}$ not hit by $S$.  Otherwise, every column intersects $S$.  Since $|S|=m-1\leq 2n-2$, at least two columns intersect $S$ in exactly one point.  We must therefore be in one of two cases:
\begin{itemize}
\item[(i)] The two columns intersect $S$ in different rows.  Then these columns with their $S$-points deleted, together with the $i^{th}$ row, form an element of $\mathcal{D}$ not hit by $S$.
\item[(ii)] The two columns intersect $S$ in the same row.  Then this row has at least two elements of $S$.  It follows that there must be another row \textbf{besides} the $i^{th}$ row that does not intersect $S$.  We can build an element of $\mathcal{E}$ out of the two rows not intersecting $S$, along with the two columns intersecting $S$ in the same row (with that point deleted from each column).  This element is not hit by $S$.
\end{itemize}
In every case, some element of $\B_2$  is not hit by $S$.  Thus $\norm{\B_2}\geq m$.  To see that $\norm{\B_2}\leq m$, let $S$ be the vertices in the first column of $G$. Since every element of $\B_2$ contains an entire row, $S$ is a hitting set.
\end{proof}

We now consider the case when $m=2n$.

\begin{proposition} The treewidth of $Y_{2n,n}$ is at least $2n-1$, with equality for some values of $n$ but not for others.
\label{prop:m=2n}
\end{proposition}

\begin{proof}  Combining Lemma \ref{lemma:mathoverflow}  with Proposition \ref{prop:m<2n}, we have that  $\tw(Y_{2n-1,n})\geq \sbn(Y_{2n-1,n})\geq 2n-1$.  Note that $Y_{2n-1,n}$ is a minor of $Y_{2n,n}$: it is obtained by deleting all the edges in a row, and then contracting $n$ vertical edges incident to that row.  Since treewidth is monotonic under graph minors \cite{hal}, we have $\tw(2n,n)\geq 2n-1$.

For $n=2$, the graph $Y_{4,2}$ is simply $Q_3$, the graph of the $3$-dimensional cube.  A tree decomposition of $Q_3$ with width $3$ was presented in \cite[Figure 4]{koy}, and is illustrated in Figure \ref{cube_decomposition}. This implies $\tw(Y_{4,2})\leq 3$. The lower bound of $2n-1$ is equal to $3$ in this case, so we have $\tw(Y_{4,2})=3$, and our lower bound is achieved.
\begin{figure}[hbt]
    \centering
    \includegraphics{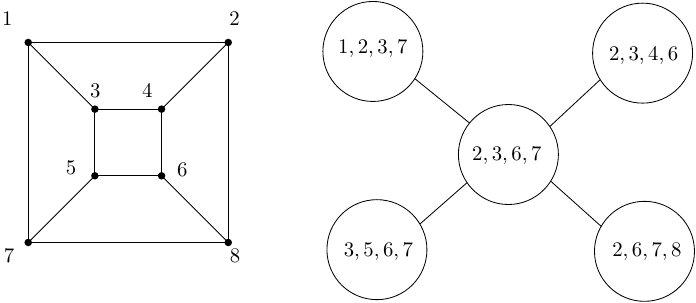}
    \caption{A tree decomposition of $Q_3=Y_{2,4}$ of width $3$}
    \label{cube_decomposition}
\end{figure}

For $n=3$, we present the following (non-strict) bramble $\mathcal{B}_3=\mathcal{F}\cup\mathcal{G}\cup\mathcal{H}$ on $Y_{6,3}$ to show that $\tw(Y_{6,3})\geq 6$.  
\begin{itemize}
    \item An element of $\mathcal{F}$ is the first column with a vertex deleted.
    \item An element of $\mathcal{G}$ is one of three sets:  the top two vertices of the second column; the middle two vertices of the second column; or the bottom two vertices of the second column.
    \item  An element of $\mathcal{H}$ consists of the last column with a vertex deleted, together with two vertices in the second column, neither in the same row as the deleted vertex.
\end{itemize}
See Figure \ref{6x3_bramble} for illustrations.
Each element of $\mathcal{B}_3$ is connected by construction, and every two elements of $\mathcal{B}_3$ touch, so it is a bramble.

\begin{figure}[hbt]
    \centering
    \includegraphics{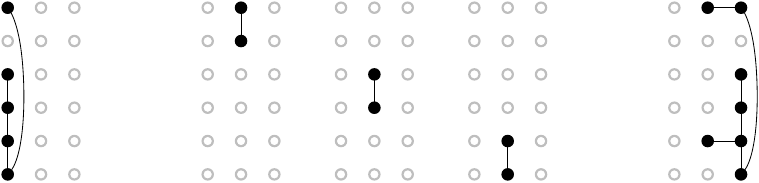}
    \caption{Some elements of the bramble $\mathcal{B}_3$.  The first is an element of $\mathcal{F}$; the next three are all three elements of $\mathcal{G}$; and the last is an element of $\mathcal{H}$.}
    \label{6x3_bramble}
\end{figure}

We claim that $\norm{\mathcal{B}_3}\geq7$.  To show this, suppose $S\subset V(Y_{6,3})$ is a hitting set of size $6$.  Then $S$ has at least two vertices in the first column, or some element of $\mathcal{F}$ is missed; and it has at least three vertices in the second column, or some element of $\mathcal{G}$ is missed.  If the sixth vertex of $S$ is also in one of the first two columns, then $S$ misses at least two vertices in the second column and all the vertices in the third column, so $S$ misses an element of $\mathcal{H}$, a contradiction.  If the sixth vertex $v$ of $S$ is in the third column, then there are exactly three vertices in the second column missed by $S$, at least two of which are in different rows from $v$; call them $u$ and $w$.  Thus the element of $\mathcal{H}$ consisting of the third column with $v$ deleted, together with $u$ and $w$, is missed by $S$, a contradiction.  We conclude that $\norm{\mathcal{B}_3}\geq7$.

We conclude that $\tw(Y_{6,3})= \bn(Y_{6,3})-1\geq7-1=6$, so the lower bound of $2n-1=5$ is not achieved for $n=3$.  (It will follow from the next proof that $\tw(Y_{6,3})=6$.)

\end{proof}

We now use our three propositions to prove Theorem \ref{theorem:stackedprism}.

\begin{proof}[Proof of Theorem \ref{theorem:stackedprism}]  First assume $m\neq 2n$.  Since the strict bramble number is a lower bound on treewidth by Lemma \ref{lemma:mathoverflow}, Propositions \ref{prop:m>2n} and \ref{prop:m<2n} imply  that $\tw(Y_{m,n})\geq \sbn(Y_{m,n})\geq \min\{m,2n\}$.  Using the upper bound from \cite{koo}, we have the equality $\tw(Y_{m,n})=\min\{m,2n\}$.

Now assume $m=2n$.  Proposition \ref{prop:m=2n} combined with the upper bound from \cite{koo} gives us that 
\[2n-1\leq \tw(Y_{2n,n})\leq 2n.\]
 By Proposition \ref{prop:m=2n}, the lower bound is sometimes but not always achieved, so the values of  $2n-1$ and $2n$ are achieved for certain values of $n$.
\end{proof}

In general, we do not know when $\tw(Y_{2n,n})$ takes on which value among $2n-1$ and $2n$.  Using the Sage command \texttt{treewidth()} \cite{sage}, we compute that $\tw(Y_{8,4})=8$, so it is possible that $\tw(Y_{2n,n})=2n$ for all $n>2$, with $Y_{4,2}$ being an anomalous case.

\section{Brambles for the toroidal grid}
\label{section:toroidal}

In this section, we consider $T_{m,n}$, the $m\times n$ toroidal graph with $m$ rows and $n$ columns. We first present a strict bramble of order $2\min\{m,n\}$ in the case when $|m-n|\geq 2$.  We then present a (non-strict) bramble of order $2\min\{m,n\}$ in the case when $|m-n|=1$.  We  combine these with upper bounds on treewidth previous work from \cite{koo} to achieve our desired results from Theorem \ref{theorem:toroidal}.

Consider $T_{m,n}$, where $m\geq n+2$.  We build a strict bramble $\mathcal{B}=\mathcal{C}\cup\mathcal{D}\cup\mathcal{E}$ on $T_{m,n}$ as follows.  In all cases, we forbid the deletion of the intersection vertex of a row and a column.

\begin{itemize}
  \item  An element of $\mathcal{C}$ is the union of one column and four rows, with one vertex removed from the column and from each row, such that no three of the vertices removed from the rows sit in the same column. 
  
  \item  An element of $\mathcal{D}$ is the union of one column and three rows, with one vertex removed from each row, such that the three removed vertices are not all in the same column.
  
  \item  An element of $\mathcal{E}$ is the union of two columns and three rows, such that the two columns each have a vertex removed (possibly in the same row), and such that the three rows each have a vertex removed, \textbf{all three of which are in the same column}.
\end{itemize}

See Figure \ref{sgg_bramble} for illustrations.

\begin{figure}[hbt]
    \centering
    \includegraphics{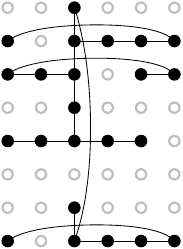}\quad\quad \includegraphics{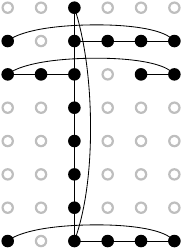}\quad\quad \includegraphics{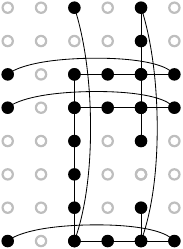}
    \caption{Elements of $\mathcal{C}$, $\mathcal{D}$, and $\mathcal{E}$, respectively}
    \label{sgg_bramble}
\end{figure}

\begin{proposition} If $m\geq n+2$, then the  family $\mathcal{B}=\mathcal{C}\cup\mathcal{D}\cup\mathcal{E}$ is a strict bramble of order $2n$ on $T_{m,n}$.
\label{prop:n<m+1}
\end{proposition}

\begin{proof}   First we show that $\B$ is a strict bramble.  Note that each element of $\mathcal{B}$ is a connected subgraph by construction.  We also need that any two elements of $\mathcal{B}$ intersect in at least one vertex.  Any element of $\B$ has at least one column (minus at most one point), which will intersect at least one of the four rows of any element of $\mathcal{C}$.  Similarly, the complete column of any element of $\mathcal{D}$ will intersect at least one row of an element of $\mathcal{D}$.  Given an element $D\in \mathcal{D}$ and an element  $E\in\mathcal{E}$, they too must intersect: even if the column of $D$ passes through the missing column of $E$, only one column of $E$ can miss all rows of $D$. Finally, given two elements of $\mathcal{E}$, at least one of the columns of one will intersect at least two rows of the other.  Thus, $\B$ is a strict bramble.

We now show that $\norm{\B}=2n$.  Let $S\subset V(T_{m,n})$ with $|S|=2n-1$.  Since $G$ has $n$ columns, at least one column intersects $S$ in no more than $1$ point. Let us call this the $i^{th}$ column. Moreover, since $n\leq m-2$, we have $|S|\leq 2m-5$.  It follows that there are at least five rows that intersect $S$ in at most $1$ point, and so at least four of them do not intersect the $i^{th}$ column at a point of $S$.

First, suppose no three of these four rows have an element of $S$ in the same column. Then the four rows together with the $i^{th}$ column, each with any point of $S$ deleted, forms an element of $\mathcal{C}$, so $S$ is not a hitting set for $\mathcal{B}$.

Otherwise, three of these four rows each contain exactly one element of $S$ in the same column. Since $|S|=2n-1$ and there are $n$ columns, one of which contains $3$ elements of $S$, we must be in one of two cases:

\begin{itemize}
\item[(i)] Two columns of $T_{m,n}$ each contain at most one element of $S$.
\item[(ii)] One column of $T_{m,n}$ contains $0$ elements of $S$, and another contains exactly $2$ elements of $S$.
\end{itemize}

If we are in case (i), we can form an element of $\mathcal{E}$ out of these two columns (with any element of $S$ removed) and three of our rows with their unique element of $S$ in a shared column.  This element of $\mathcal{E}$ does not intersect $S$, so $S$ is not a hitting set for  $\mathcal{B}$.

If we are in case (ii), then some column, say the $j^{th}$, contains no elements of $S$.  Recall that there exist five rows, each with at most one element of $S$.  Suppose not all of these rows have their point of $S$ in the same column (or that at least one of the rows has no element of $S$). Then, choosing the $j^{th}$ column together with three of our rows with the appropriate points removed, we may build an element of $\mathcal{D}$ that contains no vertex in $S$. Otherwise, the five rows have their point of $S$ in the same column.  This means that some column contains at least five elements of $S$, so the other $n-1$ columns of $T_{m,n}$ have at most $2n-6$ elements of $S$ between them.  It follows that at least two columns contain at most one element of $S$, and we are back in case (i).  Either way, we have that $S$ is not a hitting set for $\mathcal{B}$.

Since $S$ is an arbitrary subset of $V(T_{m,n})$ with $2n-1$ elements and $S$ is not a hitting set for $\mathcal{B}$, we conclude that $\norm{\B}\geq 2n$. To see that $\norm{\B}=2n$, note that the set of all vertices in two rows forms a hitting set of size $2n$.
\end{proof}

So far, we  have left out two cases of toroidal grids:  the square toroidal grid $T_{n,n}$, and the almost-square toroidal grid $T_{n+1,n}$.  As shown in \cite{koo}, $\tw(T_{n,n})$ is either $2n-2$ or $2n-1$, with the outcome varying with $n$.  We will show that a similar phenomenon occurs for $\tw(T_{n+1,n})$, except taking on values of $2n-1$ or $2n$. To start, we construct a (non-strict) bramble on $T_{n+1,n}$.  Let $\mathcal{B}=\mathcal{F}\cup\mathcal{G}$ be constructed as follows.  As usual, the intersection of a row and a column may not be deleted.

\begin{itemize}
    \item An element of $\mathcal{F}$ is a column together with a row, with one vertex deleted from the column.
    \item  An element of $\mathcal{G}$ is constructed as a column together with two rows, with a vertex  deleted from the column and from each of the two rows.
\end{itemize}

See Figure \ref{sgg_bramble_almost_square} for an illustration.

\begin{figure}[hbt]
    \centering
    \includegraphics{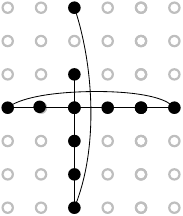}\quad\quad \includegraphics{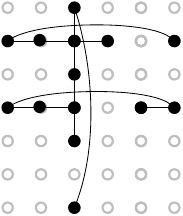}\quad\quad 
    \caption{Elements of $\mathcal{F}$ and $\mathcal{G}$, respectively}
    \label{sgg_bramble_almost_square}
\end{figure}

\begin{proposition}\label{prop:m=n+1}  The collection $\mathcal{B}$ is a bramble of order $2n$ on $T_{n+1,n}$.
\end{proposition}

\begin{proof}
First we argue that $\mathcal{B}$ is a bramble.  By construction, every element of $\B$ is connected. Let $B\in \mathcal{B}$.  We show that $B$ either shares a vertex with or touches by an edge every other element of $\mathcal{B}$.  If compared with an element $B'$ of $\mathcal{F}$, then $B$ will either share a vertex with the row of $B'$, or the column of $B$ will have its missing vertex in the row of $B'$. In the first case, $B\cap B'\neq\emptyset$, and in the latter case, the column of $B$ is still connected to the row of $B'$ with an edge.  Now compare $B$ with an element $B''$ of $\mathcal{G}$.  If $B\cap B''\neq\emptyset$, we are done.  Otherwise,  the column of $B$ does not intersect either row of $B''$. There are two rows of $B''$, and the column of $B$ is missing only one vertex, so the column of $B$ contains at least one of the vertices deleted from a row of $B''$.  This vertex is connected to $B''$ by an edge, so $B$ touches $B''$.  We conclude that $\B$ is a bramble.

Now we argue that $\norm{\B}=2n$.  Let $S\subset V(G)$ with $|S|=2n-1$.  Then some column of the graph has at most $1$ element of $S$, and either:
\begin{itemize}
    \item[(i)] some row has no elements of $S$, or
    \item[(ii)] at least three rows have at most one element of $S$.
\end{itemize}
If we are in case (i), then we can build an element of $\mathcal{F}$ from the column with at most $1$ element of $S$ and the row with no elements of $S$.  This element of $\mathcal{B}$ is not  hit by $S$, so $S$ is not a hitting set.

If we are in case (ii), then at least two of the three rows have their element of $S$ away from the column with at most $1$ element of $S$.  So, we can build a element of $\mathcal{G}$ out of those two rows and that column.  This element of $\mathcal{B}$ is not hit by $S$, so $S$ is not a hitting set.

In all cases, $S$ is not a hitting set, so $\norm{\B}\geq 2n$.  To see that $\norm{\B}\leq2n$, let $S$ consist of the first two rows of $G$.  Then $|S|=2n$, and $S$ hits each element of $\B$.  We conclude that $\norm{\B}=2n$.
\end{proof}

We are now ready to prove our theorem  on the treewidths of toroidal grids.

\begin{proof}[Proof of Theorem \ref{theorem:toroidal}]
First assume $|m-n|\geq 2$.   By Proposition \ref{prop:n<m+1} and Lemma \ref{lemma:mathoverflow}, we have $\tw(T_{m,n})\geq\sbn(T_{m,n})\geq 2\min\{m,n\}$.  Combined with the upper bound from \cite{koo}, we have $\tw(T_{m,n})=2\min\{m,n\}$.

Now assume $|m-n|=1$.  To see that $\tw(T_{m,n})\geq 2\min\{m,n\}-1$, we apply Proposition \ref{prop:m=n+1} combined with the fact that $\tw(G)= \bn(G)-1$ for any graph $G$ \cite{rs}.  The upper bound of $\tw(T_{m,n})\leq 2\min\{m,n\}$ was observed in \cite{koo}. The fact that the two values are achieved for different choices of $m$ and $n$ is established in Example \ref{example:tnn+1} below.

The case of $m=n$ was already handled in \cite{koo}.  This completes the proof.
\end{proof}

\begin{example}\label{example:tnn+1}  In this example we establish that, depending on the value of $n$, we can have either $\tw(T_{n+1,n})=2n-1$ or $\tw(T_{n+1,n})=2n$. First consider the case of $n=3$. It was shown in \cite[Proposition 5.1]{lucena} that $\tw(C_4\cart K_n)=2n-1$, and since $T_{4,3}=C_4\cart C_3=C_4\cart K_3$, we have $\tw(T_{4,3})=2\cdot3-1=5$.  

Now consider the case of $n=4$. We claim that $\tw(T_{5,4})=2\cdot 4=8$. A tree decomposition of $T_{5,4}$ with width $8$ is illustrated in Figure \ref{figure:decomposition_T54}.  We have verified computationally that this tree decomposition is minimal.  See \cite{acp} for an algorithm that can compute treewidth; we used such an algorithm that is implemented in Sage \cite{sage}. Interestingly, $T_{5,4}$ is minimal among graphs of treewidth $8$ with respect to taking minors:  removing any vertex or removing or contracting any edge drops the treewidth, as can be checked computationally.
\end{example}

\begin{figure}[hbt]
    \centering
    \includegraphics{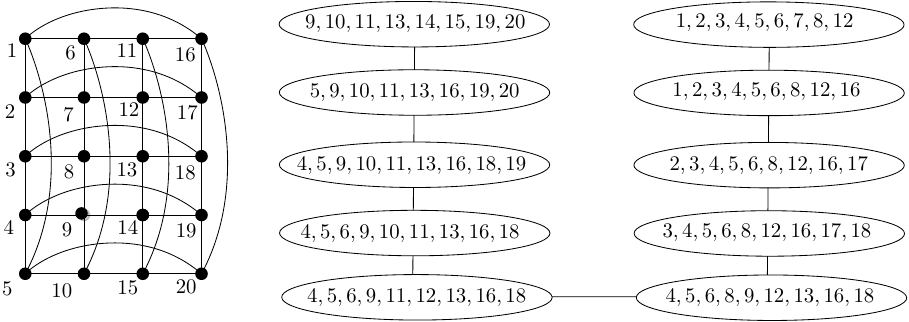} 
    \caption{A tree decomposition of $T_{5,4}$ with width $8$}
    \label{figure:decomposition_T54}
\end{figure}

As noted in \cite{koo}, it is not clear for which values of $n$ we have $\tw(T_{n,n})=2n-2$ or $\tw(T_{n,n})=2n-1$.  Similarly, it is not obvious for which values we have $\tw(T_{n+1,n})=2n-1$ or $\tw(T_{n+1,n})=2n$, and would be worth investigating in future work.

\section{Divisorial gonality}
\label{section:gonality}
We close with a brief discussion of divisorial gonality of graphs, and use Theorems \ref{theorem:stackedprism} and \ref{theorem:toroidal} to compute the gonality of all glued grids except for $T_{n,n}$, $T_{n+1,n}$, and $Y_{2n,n}$.  Divisor theory on graphs was introduced in \cite{bn} as a discrete analog of divisor theory on algebraic curves, and has been used to great effect in tropical geometry and algebraic geometry.  See \cite{vddbg} for more background, especially as it relates to treewidth.

Let $\textrm{Div}(G)$ be the free abelian group on the set of vertices of $G$.  The elements of $\textrm{Div}(G)$ are  formal integer linear combinations of elements of $V(G)$.  We call an element $D\in \textrm{Div}(G)$ a \emph{divisor} on $G$.  We can write such a $D$ as 
\[D=\sum_{v\in V(G)}a_v\cdot(v),\]
where $a_v\in \mathbb{Z}$ for all $v$.  If all $a_v\geq 0$, then we call $D$ an \emph{effective} divisor.  The \emph{degree} of a divisor $D$ is the sum of the integer coefficients appearing in $D$:
\[\deg(D)=\sum_{v\in V(G)}a_v.\]
Intuitively, one can think of a divisor as an assignment of an integer number of chips (possibly negative) to each vertex of the graph.  The degree is then the total number of chips (with negative numbers of chips cancelling positive ones).

Let $w\in V(G)$, and let $\textrm{val}(w)$ be the number of edges incident to $w$. The \emph{chip-firing move with respect to $w$} is a function from $\textrm{Div}(G)$ to $\textrm{Div}(G)$ defined by mapping the divisor $D_1=\sum_{v\in V(G)}a_v\cdot (v)$ to the divisor
\[D_2=(a_w-\val(w))\cdot (w)+\sum_{\substack{v\in V(G),\\vw\in E(G)}}(a_v+1)\cdot(v)+\sum_{\substack{v\in V(G),v\neq w,\\vw\notin E(G)}}a_v\cdot (v).\] Intuitively, $D_1$ is turned into $D_2$ by having $w$ ``chip-fire,'' meaning that $w$ donates $\textrm{val}(w)$ chips to other vertices, one to each neighbor of $w$.
We say that two divisors are \emph{equivalent} if  one can be obtained from the other by a sequence of chip-firing moves.  In Figure \ref{chipfiringexample}, the divisor $D_1=1\cdot(v_1)+1\cdot(v_2)+1\cdot(v_3)$ is equivalent to the divisor $D_2=1\cdot(v_1)+2\cdot(v_2)-1\cdot(v_3)+1\cdot v_6$ since we obtain $D_2$ from $D_1$ by performing the chip-firing move with respect to $v_3$.  Chip-firing $v_2$ then gives us $D_3$, and chip-firing $v_3$ then gives us $D_4$.  Since they differ by a sequence of chip-firing moves, the divisors $D_1$, $D_2$, $D_3$, and $D_4$ are all equivalent to one another.

\begin{figure}[hbt]
    \centering
    \includegraphics[scale=0.8]{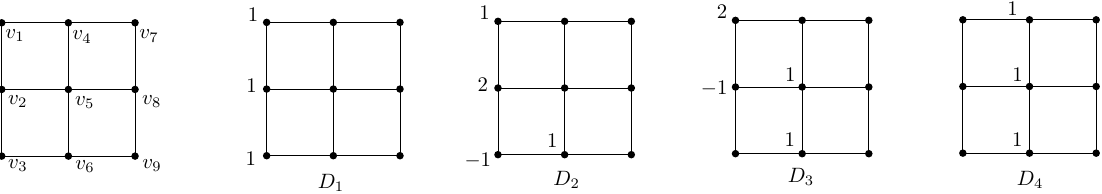}
    \caption{A labelling of the vertices of the $3\times 3$ grid graph $G_{3,3}$, followed by four equivalent divisors on $G_{3,3}$.  We move from $D_1$ to $D_2$, from $D_2$ to $D_3$, and from $D_3$ to $D_4$ by firing $v_3$, $v_2$, and $v_1$, respectively.}
    \label{chipfiringexample}
\end{figure}

Let $D\in \textrm{Div}(G)$.  The \emph{rank} of $D$, denoted $r(D)$, is the smallest integer $r\geq0$ (if it exists) such that for all effective divisors $E$ with $\deg(E)=r$, the divisor $D-E$ is equivalent to an effective divisor.
If no such $r$ exists, we define $r(D)=-1$. The \emph{divisorial gonality} (or simply \emph{gonality}) of a graph, denoted $\gon(G)$, is the smallest integer $k$ such that there exists a divisor $D$ on $G$ with $\deg(D)=k$ and $r(D)\geq1$.

We can equivalently define gonality in terms of the following chip-firing game.  Player 1 places $k$ chips on a graph $G$.  Player 2 then places $-1$ chips on a vertex of $G$.  If Player 1 can perform a series of chip-firing moves to obtain an effective divisor, then Player 1 wins; otherwise Player 2 wins.  The gonality of $G$ is then the smallest number $k$ such that Player 1 always has a winning strategy for their initial placement, so that they can win no matter on which vertex Player $2$ places $-1$ chips.

In general, computing the gonality of a graph is NP-hard \cite{gij}, so any result relating gonality to other invariants can be very useful in understanding gonality for certain classes of graphs.  Of particular use are lower bounds, such as the following result.

\begin{theorem}[Theorem 2.1 in \cite{vddbg}] For any graph $G$, $\tw(G)\leq\gon(G)$.
\label{theorem:twgon}
\end{theorem}

It is worth remarking that the related parameter \emph{pathwidth} \cite{rs1} cannot serve as a lower bound on gonality.  For instance, the gonality of any tree is $1$ \cite[Lemma 1.1]{bn2}, but the pathwidth of a tree can be arbitrarily large \cite[\textsection 1]{rs1}.  In general, treewidth and gonality can be arbitrarily far apart, even fixing treewidth:  if $2\leq k<n$, then there is a graph with treewidth $k$ and gonality at least $n$ \cite{hen}.  However, for some families of graphs treewidth and gonality coincide.  In \cite{db}, it is shown that the gonality of an $m\times n$ grid is $\min\{m,n\}$.  We now prove a similar result for most glued grids.

\begin{theorem}
\label{theorem:stackedprismgon} If $m\neq 2n$, then $\gon(Y_{m,n})=\min\{m,2n\}$.
\end{theorem}

\begin{proof}  By Theorems \ref{theorem:stackedprism} and \ref{theorem:twgon}, we have $\gon(Y_{m,n})\geq \min\{m,2n\}$.  To show that $\gon(Y_{m,n})\leq \min\{m,2n\}$, we must present a divisor $D$ of degree $\min\{m,2n\}$ with $r(D)\geq 1$. Two divisors of positive rank, one with of degree $m$ and one of degree $2n$, are described below and illustrated in Figures \ref{divisormchips} and \ref{divisor2nchips}.

For a positive rank divisor of degree $m$, let $D$ have one chip on every vertex in the leftmost column. Firing all $m$ vertices in the leftmost column (in any order) moves the chips to the right, so that they cover the next column.  Firing all the vertices on and to the left of that column moves them to the right by one column once again, and so on.  Thus, the divisor $D$ is equivalent to any other divisor consisting of one chip on each vertex of any given column.  This implies that $D$ wins the gonality game, as  wherever  Player 2 puts $-1$ chips, the column of chips can be moved to cover it.  Thus $D$ has $r(D)\geq 1$ and $\deg(D)=m$, implying $\gon(Y_{m,n})\leq m$.

\begin{figure}[hbt]
    \centering
    \includegraphics{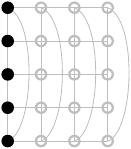}\quad\quad \includegraphics{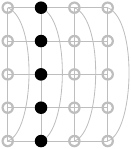}\quad\quad
    \includegraphics{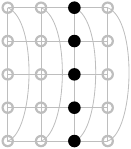}\quad\quad
    \includegraphics{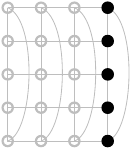}
    \caption{The divisor $D$ with $m$ chips, along with three equivalent divisors.  Each solid dot represents one chip. Moving from the first divisor to the second requires $m$  chip-firing moves; then $2m$ to move from the second to the third; then $3m$ to move from the third to the fourth.}
    \label{divisormchips}
\end{figure}

For a positive rank divisor $D'$ of degree $2n$, choose a row of $Y_{m,n}$, and let $D'$ have two chips on every vertex in that row.  Firing each vertex in this row (again, in any order) moves $n$ chips to the row above and $n$ chips to the row below.  Then firing these two rows and the row between them moves the top chips one row up and the bottom chips one row down, and so on.  Again, $D'$ is equivalent to a collection of divisors that together cover every vertex of the whole graph.  This means that by chip-firing we can cancel out the $-1$ chips placed by Player 2.  Thus $r(D')\geq 1$, and $\gon(Y_{m,n})\leq 2n$.  

\begin{figure}[hbt]
    \centering
    \includegraphics{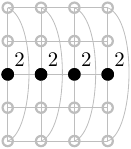}\quad\quad \includegraphics{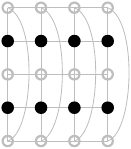}\quad\quad
    \includegraphics{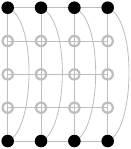}
    \caption{The divisor $D'$ with $2n$ chips, along with two equivalent divisors. Each solid dot represents one chip, unless labelled otherwise. Moving from the first divisor to the second requires $n$  chip-firing moves; then $3n$ to move from the second to the third.}
    \label{divisor2nchips}
\end{figure}

Our divisors give us that $\gon(Y_{m,n})\leq \min\{m,2n\}$, and  we conclude that $\gon(Y_{m,n})=\min\{m,2n\}$.
\end{proof}

\begin{theorem} If $|m-n|\geq 2$, then $\gon(T_{m,n})=2\min\{m,n\}$.
\label{theorem:toroidalgon}
\end{theorem}

\begin{proof}   By Theorems \ref{theorem:stackedprism} and \ref{theorem:twgon}, we have $\gon(T_{m,n})\geq 2\min\{m,n\}$.  We must now show that $\gon(T_{m,n})\leq 2\min\{m,n\}$.

For a winning divisor with $2n$ chips,   we build a divisor identical to the divisor $D'$ from the previous proof, again choosing a row and placing two chips on each vertex in that row.  Once again, we may fire rows of vertices to move the chips to cover the whole graph.  This means the divisor $D'$ wins the gonality game, and we have $\gon(T_{m,n})\leq 2n$.  Similarly, choosing a column and placing two chips on each vertex  gives a winning divisor with $2m$ chips.  This implies that $\gon(T_{m,n})\leq 2\min\{m,n\}$, so we may conclude that $\gon(T_{m,n})= 2\min\{m,n\}$.
\end{proof}

It is worth noting that the divisors from our proofs still win the gonality game for our exceptional cases $Y_{2n,n}$, $T_{n,n}$, and $T_{n+1,n}$, and so give an upper bound on gonality.  Combined with the lower bounds from treewidth, we have
\[2n-1\leq \gon(Y_{2n,n})\leq 2n,\]
\[2n-2\leq \gon(T_{n,n})\leq 2n,\] 
and
\[2n-1\leq \gon(T_{n+1,n})\leq 2n.\] 

We conjecture that in general $\gon(T_{m,n})=2\min\{m,n\}$ though it has not yet been shown that $\gon(T_{n,n}) = \gon(T_{n+1,n})= 2n$.
It is also not known what the gonality of $Y_{2n,n}$ is for general $n$.  We have computed through brute force that $\gon(Y_{4,2})=4$, and our treewidth computations for $Y_{6,3}$ and $Y_{8,4}$ imply that $\gon(Y_{6,3})=6$ and $\gon(Y_{8,4})=8$.  We conjecture that we have $\gon(Y_{2n,n})=2n$ for all $n$.

\bigskip

\noindent \textbf{Acknowledgements.}  The authors are grateful for support they received from NSF Grants DMS1659037 and DMS1347804, and from the Williams College SMALL REU program.  We also thank Jan Kyncl for suggesting the proof of Lemma~\ref{lemma:mathoverflow}.

\bibliographystyle{abbrv}
\bibliography{bibliography}

\end{document}